\documentclass[12pt,amstex]{amsart}
\usepackage{stmaryrd}

\usepackage{epsfig}
\usepackage{amsmath}
\usepackage{amssymb}
\usepackage{amscd}
\usepackage{graphicx}

\usepackage{pstricks}

\usepackage{tocvsec2}

\usepackage{hyperref}

\input xy
\xyoption{all}

\newtheorem{thm}{Theorem}[section]

\newtheorem{pro}[thm]{Property}

\newtheorem{prop}[thm]{Proposition}

\theoremstyle{definition}
\newtheorem{de}[thm]{Definition}

\theoremstyle{remark}
\newtheorem{rem}[thm]{Remark}

\numberwithin{equation}{section}

\def \T {\mathbb{T}}
\def \N {\mathbb N}

\def \Z {\mathbb Z}
\def \R {\mathbb R}

\def \F {\mathcal{F}}

\def \a {\alpha }

\def \ep {\epsilon}
\def \d {\delta}

\begin{document}
\title{Affine transformation with zero entropy and nilsystems}

\author{Zhengxing Lian}

\address{Wu Wen-Tsun Key Laboratory of Mathematics, USTC, Chinese Academy of Sciences and
Department of Mathematics, University of Science and Technology of China,
Hefei, Anhui, 230026, P.R. China, and Department of Mathematics, SUNY at Buffalo, Buffalo, NY 14260-2900, U.S.A.}
\email{lianzx@mail.ustc.edu.cn, zhengxin@buffalo.edu}



\thanks{The author is supported by NNSF of China (11171320, 11571335).}

\begin{abstract}
In this paper, we study affine transformations on tori, nilmanifolds and compact abelian groups. For these systems, we show that an equivalent condition for zero entropy is the orbit closure of each point has a nice structure. To be precise, the affine systems on those spaces are zero entropy if and only if the orbit closure of each point is isomorphic to an inverse limit of nilsystems.
\end{abstract}

\maketitle

\markboth{nilsystem}{Zhengxing Lian}

%\newpage

%\tableofcontents \settocdepth{subsection}

%\newpage

\section{Introduction}

By a ($\Z$-action) topological dynamical system, one means a pair $(X,T)$, where $X$ is a compact metric space and $T$ is a continuous selfmap of $X$. In this paper, we consider the cases that $X$ is a torus, a nilmanifold or a compact abelian group.

\medskip

For a topological group $G$ and a cocompact subgroup $\Gamma$, an {\em affine transformation} $\tau$ on $G/\Gamma$ is defined by an $\Gamma$-invariant automorphism $A$ of $G$ and a fixed element $g_0\in G$, i.e. $\tau(g\Gamma)=g_0 A(g)\Gamma$. In particular, when $G$ itself is a compact group and $\Gamma$ is trivial, the affine transformation $\tau: G\rightarrow G$ is the map $\tau(g)=g_0 A(g)$, where $g_0\in G, A\in {\rm Aut}(G)$.

\medskip

In this paper, we study affine transformations on tori, nilmanifolds and compact abelian groups, and give an equivalent condition of zero entropy for these systems. Our main result is as follows:

\begin{thm}\label{main conclusion}
Let $(X,\tau)$ be one affine transformation of the following spaces:
\begin{enumerate}
  \item  tori;
  \item  nilpotent manifolds;
  \item  compact abelian groups.
\end{enumerate}
Then $(X,\tau)$ is zero entropy if and only if for each $x\in X$, the closure of the orbit of $x$ with $\tau$ $(\overline{orb(x,\tau)},\tau)$ is isomorphic to an inverse limit of nilsystems. We might call the system $(\overline{orb(x,\tau)},\tau)$ the orbit system.
\end{thm}

%The next theorem is the most important part of Theorem \ref{main conclusion}.

%\begin{thm}\label{torus,Lie group,compact abelian group}
%Let $(X,\tau)$ be an affine transformation with zero entropy. If $X$ is one of the following spaces:
%\begin{enumerate}
%  \item  torus;
%  \item  nilpotent manifolds;
%  \item  compact abeilan groups.
%\end{enumerate}
%then for each $x\in X$, the closure of orbit system $(\overline{orb(x,\tau)},\tau)$ is isomorphic to %an inverse limit of nilsystems.
%\end{thm}

%Nilsystem is a $d$ step nilmanifold $G/\Gamma$ with its rotation $T$ that $Tx=\tau x$ for some $ \tau \in G$. Here $G$ is a d step nilpotent Lie group and $\Gamma$ is a discrete cocompact subgroup of $G$. We give the detail of the definition in Section 2.

%We also prove the conclusion:
%\begin{thm}\label{compact abelian iso}
%Let $G$ be a abelian compact group, and $\tau: X\rightarrow X, x\mapsto \a xx_0$ be an affine map with zero entropy. Then for each $x\in X$, the closure of orbit system $(\overline{orb(x,\tau)},\tau)$ is isomorphism to a $\infty$-step nilsystem(which is a limit of minimal nilsystems).
%\end{thm}

As an application, we show Sarnak conjecture holds for the zero entropy affine transformations on tori, nilmanifolds and compact abelian groups, which is first proved in \cite{LS}. According to Green and Tao's result\cite{GT2012-2}(i.e.M$\ddot{o}bius function$ is orthogonal with nilsequence), we give different methods to show this.

\medskip

The paper is organized as follows:
In Section \ref{section-pre}, we introduce some basic conceptions and results needed in this paper.
In Section \ref{section-tori}, we prove Theorem \ref{main conclusion} by three different cases. To prove this theorem, we try to analysis the returning time of each point of $X$. Then we use some theorems in \cite{HSY} to prove our results.
%Finally in Section \ref{section-Sarnak}, we give the application of the main result.

\bigskip

\noindent {\bf Acknowledgments.}
The author would like to thank Wen Huang, Song Shao and Xiangdong Ye for introducing him this topic and their constant support.

\section{Preliminaries}\label{section-pre}

In this section, we introduce some basic definitions and results needed in this paper.

\subsection{Topology entropy}%(Bowen)

Let $(X,\tau)$ be a topological dynamical system, where $\tau:X\rightarrow X$ is continuous, $X$ is a compact metric space with metric $\rho$. Let $\rho_n(x,y)=\max_{0\leq i\leq n-1}\rho(\tau^ix,\tau^iy)$. A set $E\subset X$ is $(n,\ep)$-$separated$ if for any $x\neq y\in E$, $\rho_n(x,y)>\ep$. Denote by $N(n,\ep)$ the maximum cardinality of an $(n,\ep)$-separated set. The {\em topological entropy} of $(X,\tau)$ is defined by
$$h_{top}(X,\tau)=\lim_{\ep\rightarrow 0}(\limsup_{n\rightarrow \infty}\frac{1}{n}\log N(n,\ep))$$
Let $\pi:(X,\tau)\rightarrow (Y,S)$ be a factor map, i.e. $\pi$ is a continuous surjection and $\pi \circ \tau=S\circ \pi$. Then $h_{top}(X,\tau)\geq h_{top}(Y,S)$.

\medskip

In this paper, we will need the following famous result:
\begin{thm}\label{entropy of torus}\cite{Sinai}
Let $D$ be a homomorphism of a $d$-dim torus. $\{\lambda_1,\ldots,\lambda_d\}$ are the eigenvalues of $D$. Then the entropy of $D$ is:
$$h(\T^d,D)=\sum_{j=1}^d \log|\lambda_j|_{+},$$
where $\log |\lambda_j|_{+}=\max\{\log|\lambda_j|,0\}$.
\end{thm}
\subsection{Nilsystem}
\subsubsection{Nilpotent groups} Let $G$ be a group. For $g,h\in G$, we write $[g,h]=ghg^{-1}h^{-1}$ for the commutator of $g$ and $h$ and we write $[A,B]$ for the subgroup spanned by $\{[a,b]:a\in A,b\in B\}$. The commutator subgroup $G_j,j\geq 1$, are defined inductively by setting $G_1=G$ and $G_{j+1}=[G_j,G]$. Let $d\geq 1$ be an integer. We say that $G$ is $d$-$step\  nilpotent$ if $G_{d+1}$ is the trivial subgroup.

\subsubsection{Nilmanifolds} Let $G$ be a d-step nilpotent Lie group and $\Gamma$ is a discrete cocompact subgroup of $G$, i.e. a uniform subgroup of $G$. The compact manifold $X=G/\Gamma$ is called a $d$-$step\ nilmanifold$. The group $G$ acts on $X$ by left translations and we write this action as $(g,x)\rightarrow gx$. The Haar measure $\mu$ of $X$ is the unique probability measure on $X$ invariant under this action. Let $g\in G$ and $\tau$ be the transformation $\tau(x)\rightarrow g x$ of $X$. Then $(X,\tau,\mu)$ is called a $d$-$step\ nilsystem$. See \cite{CG} for more details.

\subsubsection{Systems of  order d} We also make use of inverse limits of nilsystems and so we recall the definition of an inverse limit of systems(restricting ourselves to the case of sequential inverse limits). If $(X_i,\tau_i)_{i\in \N}$ are systems with $diam(X_i)\leq M\leq \infty$ and $\phi_i: X_{i+1}\rightarrow X_i$ are factor maps, the $inverse\ limit$ of the systems is defined to be the compact subset of $\Pi_{i\in \N}X_i$ given by $\{(x_i)_{i\in \N}:\phi_i(x_{i+1})=x_i,i\in \N\}$, which is denoted by $\lim_{\leftarrow}\{X_i\}_{i\in \N}$. It is a compact metric space endowed with the distance $\rho(x,y)=\sum_{i\in \N}\frac{1}{2^i}\rho_i(x_i,y_i)$, where $\rho_i$ is the distance of $(X_i,\tau_i)$. We note that the maps $\{\tau_i\}$ induce a transformation $\tau$ on the inverse limit.
\begin{de}\cite{HKM}\label{d-step minimial nilsystems}
 A system $(X,\tau)$ is called a $system\ of \ order\ d$, if it is an inverse limit of $d$-step minimal nilsystems.
\end{de}

\subsection{Regionally proximal relation}

Let $(X,\tau)$ be a topological dynamical system. The {\em regionally proximal relation} $RP=RP(X,\tau)\subset X\times X$ is the set which contains all the point $(x,y)$ if there are sequences $x_i,y_i\in X,n_i\in \Z$ such that $x_i\rightarrow x,y_i\rightarrow y$ and $(\tau\times \tau)^{n_i}(x_i,y_i)\rightarrow (z,z),i\rightarrow \infty$, for some $z\in X$.

\begin{de}\cite{HKM1}
Let $(X,\tau)$ be a topological dynamical system with metric $\rho$. Let $d\geq 1$ be an integer. A pair $(x,y)\in X\times X$ is said to be \textit{regionally proximal of order d} if for any $\delta>0$, there exist $x',y'\in X$ and a vector $n=(n_1,\ldots,n_d)\in \Z^d$ such that $\rho(x,x')<\d$, $\rho(y,y')<\d$, and
$$\rho(\tau^{n\cdot\ep}x',\tau^{n\cdot\ep}y')<\d \  \text{\rm for any}\  \ep \in\{0,1\}^d\setminus \{(1,1,\ldots,1)\}, $$
where $n\cdot\ep=n_1\ep_1+\ldots +n_d\ep_d$. The set of regionally proximal pairs of order $d$ is denoted by $RP^{[d]}(X)$.

It is easy to see that $RP^{[d]}$ is a closed and invariant equivalence relation.
\end{de}

\begin{thm}\cite{HKM, HSY}\label{nilfactor}
Let $(X,\tau)$ be a minimal topological dynamical system and $d\in \N$. Then $(X/RP^{[d]},\tau)$ is the maximal d-step nilfactor of $(X,\tau)$.

It means that $(X/RP^{[d]},\tau)$ is a system of order d. Also, any system of order d, which is a factor of $(X,\tau)$, is a factor of $(X/RP^{[d]},\tau)$.
\end{thm}

Now we use $RP^{[d]}$ to define $\infty$-step nilsystem\cite{DSASX}.
Observe that
$$\ldots\subset RP^{[d+1]}\subset RP^{[d]}\subset \ldots \subset RP^{[1]}=RP(X,\tau)$$
Let $RP^{[\infty]}=\bigcap_{d\geq 1}RP^{[d]}$. It follows that for any minimal system $(X,\tau)$, $RP^{[\infty]}$ is a closed invariant equivalence relation.

\begin{de}
A minimal system $(X,\tau)$ is an {\em $\infty$-step nilsystem} or a {\em system of order $\infty$}, if the equivalence relation $RP^{[\infty]}$ is trivial. i.e., coincides with the diagonal.
\end{de}

\subsection{Almost automorphic}

In \cite{HSY}, the authors gave some relationship between nilsystems and almost automorphic points.

%We say that $S\subset \Z$ is a set of $d$-$recurrence$ if for every measure preserving system $(X,\X,\mu,T)$ and for every $A\in \X$ with $\mu(A)>0$, there exists $n\in S$ such that $\mu(A\cap T^{-n}A\cap\ldots\cap T^{-dn}A)>0$.

%We say that $S\subset \Z$ is a set of $d$-$topological\  recurrence$ if for every minimal topological dynamical system(t.d.s. for short) $(X,T)$ and for every nonempty open subset $U\subset X$, there exists $n\in S$ such that $A\cap T^{-n}A\cap\ldots\cap T^{-dn}A\neq \emptyset$.

For a topological dynamical system $(X,\tau)$, a point $x\in X$ is said to be {\em almost automorphic} (AA for short) if from any sequence $\{n'_i\}\subset \Z$ one may extract a subsequence $\{n_i\}$ such that
$$\lim_{j\rightarrow \infty}\lim_{i\rightarrow \infty} \tau^{n_i-n_j}x=x.$$

A point $x\in X$ is called $d$-$step\ almost\ automorphic$ (or {\em $d$-step AA} for short) if $$RP^{[d]}[x]=\{y\in X:\ (x,y)\in RP^{[d]}\}=\{x\}.$$
A minimal topological dynamical system is call $d$-$step\ almost\ automorphic$ if it has a $d$-step AA point. Since $RP^{[d]}$ is an equivalence relation for minimal topological dynamical system, by definition it follows that for a minimal system $(X,\tau)$, it is a d-step AA system for some $d\in \N$ if and only if it is an almost one-to-one extension of its maximal d-step nilfactor.

In \cite{HSY}, one also describes AA point by the returning time of one point to it is open neighborhoods.
\subsection{Nil Bohr sets}
A subset $A$ of $\Z$ is a $Bohr$-$set$ if there exist $m\in \N$, $\a\in \T^m$ and a non-empty open set $U\subset \T^m$ such that $\{n\in \Z:n\a\in U\}$ is contained in $A$. The set $A$ is a $Bohr_0$-$set$ if additionally $0\in U$.

A subset $A$ of $\Z$ is a $Nil_d$ $Bohr_0$ $set$ if there exist a d-step nilsystem $(X,T)$, $x_0\in X$ and an open neighborhood $U$ of $x_0$ such that $N(x_0,U)=:\{n\in \Z:T^nx_0\in U\}$ is contained in $A$. Denote by $\F_{d,0}$ the family consisting of all $Nil_d$ $Bohr_0$-sets \cite{HK10}.

\medskip

Let $\F_{GP_d}$ be the family generated by the set of forms
$$\bigcap_{i=1}^k\{n\in \Z:P_i(n) \ (\text{\rm mod}\Z)\in (-\ep_i,\ep_i)\}$$
where $k\in \N,P_1,\ldots,P_k$ are generalized polynomials of degree $\leq d$ , and $\ep_i>0$.

\medskip

In this paper, we consider generalized polynomial family $GP=\bigcup_{d=0}^{\infty}GP_d$, which is defined by a sequence of family $GP_d$. Where $GP_d$ is the smallest family satisfies:
\begin{enumerate}
  \item  Any $k$, let $p_i\in \N$, $0\leq i\leq k$. If $f_1,f_2,\ldots,f_l$ with $f_i\in GP_{p_i}$ and $\sum_{i=0}^l p_i\leq d$, then $ an^{p_0}\lceil f_1\rceil \lceil f_2\rceil\ldots \lceil f_k\rceil \in GP_d$. Where $\lceil x\rceil $ is the smallest integer such that $\lceil x\rceil \geq x$.
  \item  If $g,h\in GP_d$, then $gh,g\pm h, cg,\lceil g\rceil \in GP_d$. Where $c\in \R$
\end{enumerate}
And $G_0=\R$.

If $f\in GP_d$, we say the degree of $f$ $deg(f)\leq d$.
\begin{thm}\cite{HSY}\label{F_d=F_GPd}
Let $d\in \N$. Then $\F_{d,0}=\F_{GP_d}$.
\end{thm}
\begin{thm}\cite{HSY}\label{main conclusion fo HSY}
Let $(X,T)$ be a minimal topological dynamical system, $x\in X$ and $d\in \N\cup\{\infty\}$. Then the following statements are equivalent:

(1) $x$ is d-step AA point.

(2) $N(x,V)\in \F_{d,0}$ for each neighborhood $V$ of $x$.

\end{thm}

\section{The proof of Theorem \ref{main conclusion}} \label{section-tori}

Let $(X,\tau)$ be one affine transformation of the following spaces:
\begin{enumerate}
  \item  tori;
  \item  nilpotent manifolds;
  \item  compact abelian groups.
\end{enumerate}
We want to show that $(X,\tau)$ is zero entropy if and only if for each $x\in X$, the orbit system $(\overline{orb(x,\tau)},\tau)$ is isomorphic to an inverse limit of nilsystems. The difficult part is to show if $(X,\tau)$ is zero entropy, then for each $x\in X$, the orbit system $(\overline{orb(x,\tau)},\tau)$ is isomorphic to an inverse limit of nilsystems. Since the proofs for tori, nilpotent manifolds, and compact abelian groups are quite different, we will prove them separately. Before that, we show the easy direction of the proof of Theorem \ref{main conclusion}.

\subsection{The easy part of the proof of Theorem \ref{main conclusion}}\

\medskip
Now we assume that for each $x\in X$, the orbit system $(\overline{orb(x,\tau)},\tau)$ is isomorphic to an inverse limit of nilsystems. We will show that the entropy of the system is zero. In fact, we show that $(X,\tau)$ is distal. Since each distal system has zero entropy, we finish the proof.

In fact, if for each point $x\in X$, the closure of $orb(x,\tau)$ is an inverse limit of nilsystems, then $X$ should be distal. That is, if $y,y'\in \overline{orb(x,\tau)}$, since $(\overline{orb(x,\tau)},\tau)$ is distal, then $(y,y')$ is distal.

If $y\in \overline{orb(x,\tau)}$ and $y'\in \overline{orb(x',\tau)}$ for two different orbits with $\overline{orb(x,\tau)}\cap\overline{orb(x,\tau)}=\emptyset$, since $\overline{orb(x,\tau)}$ and $\overline{orb(x',\tau)}$ is closed, $\rho(y,y')\geq \min\{\rho(z,z')\ z\in\overline{orb(x,\tau)},z'\in\overline{orb(x,\tau)}\}=c_{x,x'}>0$ for some fix number $c_{x,x'}$. Then $(X,\tau)$ is distal.

\subsection{Results on tori and nilpotent manifolds}

The main aim of this section is to prove the following results:

\begin{thm}\label{torus,Lie group,compact abelian group nilpotent}
Let $(X,\tau)$ be an affine transformation with zero entropy. If $X$ is one of the following spaces:
\begin{enumerate}
  \item  tori;
  \item  nilpotent manifolds;

\end{enumerate}
then there is a fixed $d$, such that for each $x\in X$, an open neighborhood $U$ of $x$ and its return time $N(g,U)=\{n:\tau^n x\in U\}$, we have $N(x,U)\supset \bigsqcup_{j=1}^b N_j$($\bigsqcup$ means disjoint union) for some $b\in \N$ with $N_j \in \F_{GP_d}$ and $N_i\cap N_j=\emptyset$ for each $i\neq j$.
\end{thm}

\begin{thm}\label{iso of torus and nilmanifold}
Let $(X,\tau)$ be an affine transformation with zero entropy. If $X$ is one of the following spaces:
\begin{enumerate}
  \item  tori;
  \item  nilpotent manifolds;
\end{enumerate}
  then for each $g\in G$, the orbit system $(\overline{orb(g,\tau)},\tau)$ is isomorphism to a system of order $d$.
\end{thm}
%Where a system of order $d$ is a inverse limit of $d$ step nilsystem.

\subsection{Tori}

A $d$-dim torus is $\T^d=\R^d/\Z^d$. An integer-valued $d\times d$ matrix $A$ can be seen as a homomorphism of $\R^d$ and $\Z^d$, and it induces a homomorphism on $\T^d$. A map $\tau:\T^d\rightarrow \T^d$ is an affine transformation if there are a matrix $A\in \Z^{d\times d}$ and a fixed point $b_0\in \T^d$ such that $\tau(b) = Ab + b_0$.

Note that if the integer-valued $d\times d$ matrix $A$ is a zero entropy homomorphism, then all the eigenvalues of $A$ are unity roots. In fact, according to Theorem \ref{entropy of torus}, if the integer-valued $D$ is a zero entropy homomorphism, then $\Pi_{j=1}^d \lambda_j$ is a non zero integer with all $|\lambda_j|\leq 1$, which means $|\lambda_j|=1$. And then by the following Kronecker Lemma we have that all the $\lambda_j$ are unity roots.

\begin{thm}[pp 108, Lemma (a)]\cite{Hecke}\label{roots}
If $\{\lambda_j\}_{j=1}^d$ satisfies $|\lambda_j|=1$, all the $j$, and
$$p(x)=\prod_{j=1}^d (x-\lambda_j)$$
are integral coefficicents polynomial, then all the $\lambda_j$ are unity roots.
\end{thm}

\medskip

Let
$$\rho(x,y)=\max_{1\le i\le d}|x_j-y_j|, \ \ x=(x_j)_j,y=(y_j)_j\in\T^d$$
be a metric of $\T^d$. Now we begin to consider the return time $N(a,\tau,\ep)$ of arbitrary $a=(a_j)_j\in \T^d$, where $N(a,\tau,\ep)=\{n \in \Z:\rho(\tau^na,a) <\ep\}$.

\begin{prop}\label{tori}
Let $X=\T^d$, $\tau$ be a affine transformation of $X$ with zero entropy, and $a=(a_j)_j\in X=\T^d$. Then for any $\ep$, $N(a,\tau,\ep)$ include some $d$-degree generalized polynomial return time. i.e.
$$N(a,\tau,\ep)\supset\bigcup_{r=0}^{b-1}(\bigcap_{j=1}^d\{tb+r,t\in N_{rj}\}),$$
where each $N_{rj}$ is a return time of a polynomial function $f_{rj}$ such that $N_{rj}=\big{\{}t,|\{f_{rj}(t)\}-a_j|<\ep\big{\}}$ for polynomials $f_{rj}$ with degree less than $d+1$.
\end{prop}

\begin{proof}
Let $\tau:\T^d\rightarrow \T^d$, $(x_j)^T\rightarrow A(x_j)^T+(\a_j)^T$ be an affine transformation with zero entropy. Where $\a=(\a_j)\in \T^d$ is a fix point. Since $\tau$ has zero entropy, then $A$ is a $d\times d$ integral matrix. Since Kronecker Lemma, all the eigenvalues of $A$ is roots of unity. That means there are some $B$ and $P$ such that $A=P^{-1}BP$, where $B$ is a upper triangular matrix, and all the entries in the diagonal of $B$ are roots of unity.

First we have
\begin{equation*}
\begin{split}
\tau^na & =A^n(a_j)^T+\sum_{i=0}^{n-1}A^i(\a_j)^T\\
 & = P^{-1}B^nP(a_j)^T+P^{-1}(\sum_{i=0}^{n-1}B^i)P(\a_j)^T.
\end{split}
\end{equation*}
Since $B$ is a upper triangular matrix, and all the diagonal elements of $B$ is roots of unity, there is $b$ such that $B^{b}$ is a upper triangular matrix with all the diagonal elements is $1$. Denote $C=B^{b}$. Now for $n=tb+r$, $0\leq r\leq b-1$,
$$\tau^na=P^{-1}C^tB^rP(a_j)^T+P^{-1}(\sum_{i=0}^{b-1}B^i)(\sum_{k=0}^{t-1}C^k)P(\a_j)^T+P^{-1}(\sum_{i=0}^{r-1}B^i)C^tP(\a_j)^T.
$$
For the fix $r$, we consider if $\rho(a,\tau^{tb+r}a)<\ep$. Let $C=I+D$, here $D$ is strictly upper triangular matrix, $D^d=0$. Then $C^t=\sum_{i=0}^{d-1}\frac{t!}{i!(t-i)!}D^i$. That means each element of $C^t$ is a polynomial of $t$, since $r$ is fixed, each element of $P^{-1}C^tB^rP$, $P^{-1}(\sum_{i=0}^{b-1}B^i)(\sum_{k=0}^{t-1}C^k)P$, $P^{-1}(\sum_{j=0}^{r-1}B^j)C^tP$ is a polynomial of $t$; that means each coordinate of $\tau^{tb+r}a$ is a polynomial of $t$ with degree no more than a fix $d\in \N$:
$$\tau^{tb+r}a=(f_{rj}(t))_{1\leq j\leq d}.$$
And all $f_{rj}$ are no more than $d$ degree. So for $N_{rj}=\big{\{}t,|\{f_{rj}(t)\}-a_j|<\ep\big{\}}$, here $\{f_{rj}(t)\}=f_{rj}(t)\ ({\rm mod} 1)$,
$$N(a,\tau,\ep)=\bigcup_{r=0}^{b-1}(\bigcap_{j=1}^m\{tb+r,t\in N_{rj}\}).$$
The proof is completed.
\end{proof}

\subsection{Nilpotent Lie groups}

Now we begin to prove Theorem \ref{torus,Lie group,compact abelian group nilpotent} for the nilpotent Lie group case.
\begin{prop}\label{linear map}
Let \(A\) be a continuous surjective homomorphism of $k$-step nilpotent Lie group \(G\), and let $G_i=[G_{i-1},G]$ with $G_0=G$. Then $A$ induces naturally surjective homomorphisms \(A_i\):
$$A_i:G_i/G_{i+1}\rightarrow G_i/G_{i+1},$$
and \(A_i\) is invertible linear mapping.

We call \(\{A_1,A_2,\ldots, A_k\}\) the {\em linear part} of \(A\).
\end{prop}

 Since each nilpotent Lie group can be imbedded into a upper triangular matrix group, in this paper, we only analysis the matrix nilpotent Lie group.\footnote{We can see this conclusion from Ado theorem. Also we know that each $G/\Gamma$ can be seen as a submanifold of $\hat{G}/\hat{\Gamma}$ such that $\hat{G}$ is connect Lie group. Then we use the linear property of a Lie algebra and transformation between Lie algebra and Lie group to see this conclusion.}
%\begin{prop}Let $G\subset M_n(\C)$  be a upper triangular matrix group with all diagonal elements being 1. If $\tau$ is a affine for $G$, then there is a fix $d$ that for each $g\in G$, we have each element of $\tau^n(g)$ is a polynomial of $n$ with degree no more than $d$.
%\end{prop}
%\begin{prop}\label{Lie polynomial}
%Let $G$ be a connected nilpotent Lie Group. For a fixed $r$, we have that
%$$\tau^{tb+r}(g_0)=\exp(f_{r1}(t)X_1)\ldots \exp(f_{rn}(t)X_n),$$
%where $f_i$ is a polynomial for each $i$ and all $X_i$ are Mal'cev basis of the corresponding lie algebra of $G$.
%\end{prop}
%\begin{proof}
%\end{proof}
We could write a element of $k$-step matrix nilpotent Lie group G as:
\[\begin{pmatrix}
1    & x_{1,1} & x_{2,1} & ... & x_{k-1,1}& x_{k,1}  \\
0    & 1       & x_{1,2} & ... & ...      &x_{k-1,2} \\
0    & 0       & 1       & ... & ...      & ...      \\
...  & ...     & ...     & ... & ...      & ...      \\
0    & ...     & ...     & ... & 1        & x_{1,k}  \\
0    & ...     & ...     & ... & 0        & 1
\end{pmatrix}
\]
In the following part of this paper, we write a element of G as \((x_{i,j})_{1\leq i \leq k,1\leq j \leq k+1-i}\), or just write it as \((x_{i,j})_{i,j}\) if there is no confusion. Group operation is the matrix operation.
And we write the element of \(G_i/G_{i+1}\) as \((x_{i,1},x_{i,2}...x_{i,k+1-i})\) or \(\overrightarrow{x_i}\).
Here we would define a new kind of degree of a polynomial. The polynomial is define on the matrix elements of of a element of nilpotent Lie group $G$.

\begin{de}
For \((x_{i,j})_{1\leq i \leq k,1\leq j \leq k+1-i}\), let $$W=\{\prod_{1\leq i \leq k,1\leq j \leq k+1-i}x_{i,j}^{\a_{i,j}}:\a_{i,j}\in \N\}.$$ And define $L$-$polynomial\;degree$ \(l:\;W\rightarrow\N\) such that
$$l\Big(\prod_{1\leq i \leq k,1\leq j \leq k+1-i}x_{i,j}^{\a_{i,j}}\Big)=\sum_{1\leq i \leq k,1\leq j \leq k+1-i}2^{i-1}\a_{i,j}$$
If $f$ is a constant, $l(f)=0$. For the polynomial \(f\) defined on the terms of \((x_{i,j})_{1\leq i \leq k,1\leq j \leq k+1-i}\), \(f=c_1w_1+...+c_tw_t,\;c_i\in\R,\;w_i\in W\), we define \(l(f)=\max_{i=1,2...t}{l(w_i)}\). Also it is clear that $l(fg)\leq l(f)+l(g)$ and $l(f+g)\leq \max\{l(f),l(g)\}$.
\end{de}

Here we give the conclusion of matrix nilpotent Lie group G.
\begin{prop}\label{finite degree of homomorphism of nilpotent group}
Let $G$ be a k-step matrix nilpotent Lie group. $A$ is a surjective continuous homomorphism of G. For any \(g=(x_{i,j})_{1\leq i \leq k,1\leq j \leq k+1-i}\), \(A((x_{i,j})_{i,j})=(\varphi_{i,j}(g))_{i,j}\), we have:
\begin{equation}\label{equation of nilLigroup}
\varphi_{i,j}(g)=A_{i,j}(x_{i,1},x_{i,2}...x_{i,k+1-i})+f_{i,j}(x_{1,1},...x_{1,k},x_{2,1}...x_{2,k-1}...x_{i-1,1}...x_{i-1,k+2-i})\;\;
\end{equation}
where \(A_{i,j}\) is linear functional, $f_{i,j}$ is a polynomial such that its L-polynomial degree \(l(f_{i,j})\) is no more than \(2^{i-1}\).
\end{prop}
\begin{rem}
We might use Lie algebra to prove a similar proposition. But to make the degree more exact, we use a complex proof.
\end{rem}
\begin{proof}
It is clear that for each \(G_i\), if \(A\) is surjective, \(AG_i=G_i\). Then the first part of (\ref{equation of nilLigroup}) is clear.(i.e. $A_{i,j}$ is linear.)

 Now we would prove that \(f_{i,j}\) is polynomial with \(l(f)\leq 2^{i-1}\). Because \(A\) is homomorphism, \(A(gg')=A(g)A(g')\). we have:
 \[(\varphi_{i,j}(g))(\varphi_{i,j}(g'))=(\varphi_{i,j}(gg')).\]
 That means
\begin{equation}\label{proof of nilLigroup}
\varphi_{i,j}(gg')=\Sigma_{s=0}^{i}\varphi_{s,j}(g)\varphi_{i-s,s+j}(g'). \;\;\;\;  \end{equation}

where $\varphi_{0,j}(g)=1$.

  Let $\{1,\ldots,k\}$ act on $G$ such that $t(g)$ be the matrix in \(G\) and
  $$t(g)=(x_{i,j,(t)})_{i,j} \;\;x_{t,j,(t)}=x_{t,j},\;\;x_{i,j,(t)}=0 \;\;\text{if} \;\;i\neq t.$$
 we have $t(g)\in G_t$, $g=1(g)2(g)...k(g)$, and let \(h_t(g)=1(g)...t(g),\;\;h_0(g)=I\). Since \(AG_i=G_i\), $\varphi_{i,j}(t(g))=0$ if \(t>i\geq 1\). Then (\ref{proof of nilLigroup}) gives the result that for \(i>0\) $$\varphi_{i,j}(h_{t+1}(g))=\Sigma_{s=0}^{i}\varphi_{s,j}(h_t(g))\varphi_{i-s,s+j}((t+1)(g))
 =\varphi_{i,j}(h_t(g))\;\;if\;\;t\geq i .$$
Which means $\varphi_{i,j}(g)=\varphi_{i,j}(h_i(g))$. And we have:
\begin{equation} \label{proof of nilLigroup 2}
 \begin{split}
   & \quad \ \varphi_{i,j}(h_i(g))\\ & = \Sigma_{s=0}^{i}\varphi_{s,j}(h_{i-1}(g))\varphi_{i-s,s+j}(i(g))\\
     & =\varphi_{i,j}(h_{i-1}(g))+\varphi_{i,j}(i(g))\\
     & =\Sigma_{s=0}^{i}\varphi_{s,j}(h_{i-2}(g))\varphi_{i-s,s+j}((i-1)(g))+\varphi_{i,j}(i(g))\\
     & =\varphi_{i,j}(h_{i-2}(g))+\varphi_{i-1,j}(h_{i-2}(g))\varphi_{1,j}((i-1)g)+\varphi_{i,j}((i-1)(g))+\varphi_{i,j}(i(g))\\
     &= \ldots\\
     &=\Sigma_{s=1}^i\varphi_{i,j}(s(g))+\Sigma_{s=1}^{i-1}\Sigma_{r=1}^{s}\varphi_{i-r,j}(h_{i-1-s}(g))\varphi_{r,j-r}((i-s)(g)).
 \end{split}
 \end{equation}
For \(i=1\), \(l(f_{1,j})\leq 1\) since $G/G_1$ is a linear space.  (\ref{equation of nilLigroup}) is right for \(i=1\). Now we make the induction for \(1,2,\ldots ,i-1\). i.e. for $s\leq i-1$,
$$l(\varphi_{s,j}(g))=l(\varphi_{s,j}(h_s(g)))\leq 2^{s-1}$$
By the induction, $l(\varphi_{i-r,j}(h_{i-1-s}(g))\varphi_{r,j-r}((i-s)(g)))\leq 2^{i-r-1}*2^{r-1}\leq 2^{i-1}$.

Claim: \(\varphi_{i,j}(s(g))\) is a polynomial on $x_{i,j}$ with \(l(\varphi_{i,j}(s(g)))\leq 2^{i-1}\), where \(s\leq i\).

 If the claim is right, \black according to and (\ref{proof of nilLigroup 2}), \(\varphi_{i,j}(h_{i-1}(g))\) satisfies \(l(\varphi_{i,j}(h_{i-1}(g)))\leq 2^{i-1}\). Then $l(\varphi_{i,j}(g))\leq 2^{i-1}$. Hence this proposition is right.

Now we prove the claim. For a fix $s$,
%,We need another induction for \(i\), here $i\geq s$. When $i=s$, since $A_s$ is linear functional on $G_i/G_{i+1}$, $l(\phi_{i,j}(s(g)))\leq 2^{i-1} $ is clear.
we denote $(y_1,y_2,\ldots,y_{x+1-s})_{s}$ being the element $\hat{g}=(\hat{x}_{i,j})_{i,j}\in G$ such that
$$(\hat{x}_{s,1},\ldots,\hat{x}_{s,k+1-s})=(y_1,y_2,\ldots,y_{k+1-s})\ and \ x_{i,j}=0\ for\ i\neq s.$$
According to the operation of matrix, $(0,1,0,\ldots,0)_{s}*(1,0,\ldots,0)_{s}=(1,1,0,\ldots,0)_{s}$, here each coordinate corresponding a matrix. What is more, we have:
$$(y_1,y_2,\ldots,y_{k+1-s})_s=(0,\ldots,0,y_{k+1-s})_s*(0,\ldots,0,y_{k,s},0)_s\ldots,(y_1,0,\ldots,0)_s .$$
For $1\leq s\leq k$, denote \(\gamma_{r,s}=(0,\ldots,0,x_r,0,\ldots,0)_s\in G\) corresponding to $g=(x_{i,j})_{i,j}\in G$. Then $s(g)=\gamma_{k+1-s,s}\ldots \gamma_{1,s}$.
By (\ref{proof of nilLigroup}),
$$\varphi_{i,j}((x_1,x_2,\ldots,x_{k+1-s})_s)=
\Sigma_{r=0}^{i}\varphi_{r,j}(\gamma_{k+1-s})\varphi_{i-r,j+r}((x_1,x_2,\ldots,x_{k-s},0)_s).$$

According to the induction, if we want to prove $l(\varphi_{i,j}((x_1,x_2,\ldots,x_{k+1-s})_s))\leq 2^{i-1}$, we only need to prove
$$l(\varphi_{i,j}((x_1,x_2,\ldots,x_{k-s},0)_s)\leq 2^{i-1}$$
and
$$l(\varphi_{i,j}(\gamma_{k+1-s}))\leq 2^{i-1}.$$
\black
We use (\ref{proof of nilLigroup}) again and again for $s(g)=\gamma_{k+1-s,s}\ldots \gamma_{1,s}$, then we just need to prove $l(\varphi_{i,j}(\gamma_r))\leq 2^{i-1}$ for each $r$. Let \(e_r=(0,\ldots,0,1,0,\ldots,0)_s\). Then $e_r$ generated a one-dim sub Lie group of $G$. We know $e_r^n=ne^r$ for positive integer $n$, which means
 $$A(ne_r)=A(e_r^n)=(A(e_r))^n.$$
 $\gamma_r=e_r^{x_r}$, where $e_r^{x}$ is defined by $q_n\in \mathbb{Q},q_n\rightarrow x$ such that $e_r^{x}=\lim_{n\rightarrow \infty}e_r^{q_n}$, and $e_r^{\frac{1}{t}}$ is the matrix $C$ satisfying $C^t=e_r$. For $e_r\in G$, the matrix $C$ is unique.

Then since $A$ is continuous, $(\varphi_{i,j}(\gamma_r))_{i,j}=A(e_r^{x_r})=(Ae_r)^{x_r}$. Then for $(Ae_r)^{x_r}$, since $B=A(e_r)=(b_{i,j})_{i,j}\in G$, then each matrix element of $B^{x_r}$ is a function of $x_r$. What is more, since the diagonal element of $B$ is $1$, $\varphi_{i,j}(\gamma_r)$ is a polynomial of $x_r$. Also $b_{t,j}=0$ for $1\leq t< s$. Since $A(G_s)\subset G_s$, we know that the degree $x_r$ in the polynomial $\varphi_{i,j}(\gamma_r)$ of $x_r$ is no more than $2^{i-s}$, which means that $l(\varphi_{i,j}(\gamma_r))\leq 2^{i-1}$.\black

So the claim is right. The whole proof is completed.
%  What is more, we can define \(B^{\frac{m}{n}}=(p_{i,j}(\frac{m}{n}e_r))_{i,j}\) for \(B\) and \(B^{\frac{m}{n}}\) in \(G\), according to (\ref{proof of nilLigroup 3}), it satisfies \((B)^c=(p_{i,j}(ce_r))_{i,j}\) if \(c=\frac{m}{n}\). As the Lie Group generated by \(B\) is unique, that would give the result that \(q_{i,j}(\frac{m}{n}e_r)=p_{i,j}(\frac{m}{n})\) is polynomial with \(2^{i-1}\) L polynomial degree. Since \(A\) is continuous, (\ref{proof of nilLigroup 3}) is right for any \(c\in \R\). That means, the claim is
%right for any \(i\),
\end{proof}

Now for the nilpotent manifold $G/\Gamma$ and its zero entropy affine transformation, we prove Theorem \ref{torus,Lie group,compact abelian group nilpotent}.

For an affine transformation of $G/\Gamma$, $\tau: g\rightarrow \a A(g)$ with $\a\in G$ and $A$ being an automorphism of $G/\Gamma$, which induced by a homomorphism $A$ of $G$. $A$ also keep $\Gamma$. $\tau$ has zero entropy if and only if $A$ has zero entropy. In this case, each $A_i$ in Proposition \ref{linear map} is zero entropy (each $A_i$ also induce a automorphism on $G_i/\Gamma_i$).

And for each $i$, we have $A_i=Q_i^{-1}B_iQ_i$, where $B_i$ is a upper triangular matrix and all the diagonal elements of $B_i$ are roots of unity. Also there is a matrix $P$ such that $B=P^{-1}AP$ with $B_i$ is the linear part of $B$. Let $b$ be the integer such that all the diagonal elements of $B_i^{b}$ is $1$ for each $i$. Let $C=A^{b}$. Then for $n$, we have $n=tb+r$ and $A^n=C^tA^r$ where $0\leq r \leq b-1$.

\subsection{Proof of Theorem \ref{torus,Lie group,compact abelian group nilpotent}}

We need a metric of $G/\Gamma$ which can generate the topology. One simple metric $\rho$ is the maximum distance such that for $g\Gamma,g'\Gamma\in G/\Gamma$ with $g=(x_{i,j})_{i,j}$ and $g'=(x'_{i,j})_{i,j}$ $x_{i,j},x'_{i,j}\in [0,1)$, $\rho(g\Gamma,g'\Gamma)=\max_{i,j}|x_{i,j}-x'_{i,j}|$. For each $g\Gamma\in G/\Gamma$, we only need to consider the open neighborhood $U_\ep=\{h+\Gamma,\rho(h\Gamma,g\Gamma)<\ep\}$.

We consider the affine transformation $\tau(g\Gamma)=g_0A(g\Gamma) $ on $G/\Gamma$, where $A$ is a zero entropy homomorphism on $G$(it induce a map on $G/\Gamma$)and fix element $g_0\in G$. Also $A$ satisfies $A\Gamma\subset \Gamma$. Here we denote $g_0=(y_{i,j})_{i,j}$.

Let $b$ be the integer such that all $A_i$(the linear part of $A$) satisfies that the eigenvalues of $A_i^b$ are $1$(as the discussion in torus case).

 Claim: Each matrix element of $A^{nb+r}(g)=(\phi_{i,j}(n,g))_{i,j}$ is a finite degree polynomial of $n$. The degree of $\phi_{i,j}(n,g)$ is less than some $M$, which is independent of $g$.

The claim is right since Proposition \ref{finite degree of homomorphism of nilpotent group} and $A^{nb+r}(g)=(A^b)^{n}(g)A^{r}(g)$. We know that $B=A^b$ is still a group automorphism. If we write $B(g)=(\varphi_{i,j}(g))_{i,j}$, then
$$\varphi_{i,j}(g)=B_{i,j}(x_{i,1},x_{i,2}...x_{i,k+1-i})+h_{i,j}(x_{1,1},...x_{1,k},x_{2,1}...x_{2,k-1}...x_{i-1,1}...x_{i-1,k+2-i})$$
Where $A_i^b=B_i=(B_{i,1}\ldots,B_{i,k+1-i})$, which means that $B_i$ is the matrix with all eigenvalues being 1. We write $h_{i,j}(x_{1,1},...x_{1,k},x_{2,1}...x_{2,k-1}...x_{i-1,1}...x_{i-1,k+2-i})=H_{i,j}(g),$ $(x_{i,1},x_{i,2}...x_{i,k+1-i})=\vec{x}_i$, $\vec{H}_i(g)=(H_{i,1}(g),\ldots,H_{i,k+1-i}(g))$. Then
$$(\phi_{i,1}(n),\ldots,\phi_{i,k+1-i}(n))=B_i^n(\vec{x}_i)+\sum_{j=0}^{n-1}B_i^{n-1-j}(\vec{H}_i(B^jg))$$
We make an induction for $i$ to show that $\phi_{i,j}(n)$ are polynomials of $n$.

$i=1$ is clear since the tori case. Now we assume that for $1\leq t\leq i-1$, $\phi_{t,j}(n)$ are polynomials of $n$. Then $H_{i,j}(B^mg)$ are polynomials of $m$ since $h_{i,j}(*)$ are polynomials and $H_{i,j}(B^mg)$ is a polynomial of $\phi_{t,j}(m)$ with $1\leq t\leq i-1$.

Also if $\vec{z}=(z_1,\ldots,z_{k+1-i})$ is fixed, each coordinate of $B_i^m(\vec{z})$ are polynomials of $m$ since the discussion of tori. what is more, for fix $m$, each coordinate of  $B_i^m(z_1,\ldots,z_{k+1-i})$ is a linear combination of $z_1,\ldots,z_{k+1-i}$. That means, each coordinate of $B_i^{n-1-j}(\vec{H}_i(B^jg))$ are polynomials of $n$ and $j$. Then each coordinate of $B_i^{-j}(\vec{H}_i(B^jg))$ are polynomials of $j$, we write it as $B_i^{-j}(\vec{H}_i(B^jg))=(v_1(j),\ldots,v_{k+1-i}(j))$. Where $v_s(j)$ are polynomials of $j$.

Then $$\sum_{j=0}^{n-1}B_i^{n-1-j}(\vec{H}_i(B^jg))=B_i^{n-1}(\sum_{j=0}^{n-1}v_1(j),\ldots,\sum_{j=0}^{n-1}v_{k+1-i}(j))$$
Since for arbitrary fix $m\in \N$, $\sum_{j=0}^{n-1}j^m$ are polynomials of $n$. Then $\sum_{j=0}^{n-1}v_s(j)$ are polynomials of $n$ for $1\leq s \leq k+1-i$. Then we have $\phi_{i,j}(n)$ are polynomials of $n$ for each $j$. That is the claim.

Now we consider the zero entropy affine $\tau$.

$$\tau^{nb+r}(g\Gamma)=A(g_0)\ldots A^{nb+r-1}(g_0)A^{nb+r}(g_0)A^{nb+r}(g)\Gamma.$$
We consider $A(g_0)\ldots A^{nb+r-1}(g_0)A^{nb+r}(g_0)A^{nb+r}(g)$. Notice that $A^{m}(g_0)$ is a upper triangle matrix with the dialog element being $1$.The claim tells us that

$A^{(n-1)b+r+1}(g_0)\ldots A^{nb+r}(g_0)$ is a upper triangle matrix with diagonal element being 1 that each matrix element of it is a finite degree polynomial of $n$. Then according to matrix operation, we know that each matrix element of
$A(g_0)\ldots A^{nb+r-1}(g_0)A^{nb+r}(g_0)$ is a finite degree polynomial of $n$.

Now we denote $A(g_0)\ldots A^{nb+r-1}(g_0)A^{nb+r}(g_0)A^{nb+r}(g)=(p_{i,j,r}(n))_{i,j}$, where each $p_{i,j,r}(n)$ is a polynomial with degree less than some $M_0$ independent with $g$. We want to find $g_{n,r}\Gamma=(p_{i,j,r}(n))_{i,j}\Gamma$ such that $g_n=(q_{i,j,r}(n))_{i,j}$ with $q_{i,j,r}(n)\in [0,1)$.

For $i=1,$ $q_{1,j,r}(n)=p_{1,j,r}(n)(\mod 1)$. We want to find $\gamma_{n,r}=(y_{i,j,r}(n))_{i,j}\in \Gamma\subset M_{k+1}(\Z)$ such that $g_n\gamma_{n,r}=(p_{i,j,r}(n))_{i,j}$, that means
$$p_{i,j,r}(n)=\sum_{s=0}^iq_{s,j,r}(n)y_{i-s,s+j,r}(n).$$
Here $q_{0,j',r}=1,y_{0,j',r}=1$ for all $j'$. Then $y_{1,j'}=[p_{1,j'}]$ are finite degree generalized polynomial for all $j'$.

Inductively, we assume $q_{1,j',r},\ldots,q_{i-1,j',r},\gamma_{1,j',r},\ldots,\gamma_{i-1,j',r}$ are generalized polynomial with finite degree for all $j'$. Then
$$q_{i,j,r}+y_{i,j,r}=p_{i,j,r}-\sum_{s=1}^{i-1}q_{s,j,r}y_{i-s,s+j,r}$$
Since $\gamma_{n,r}\in \Gamma$,
$$y_{i,j,r}=[p_{i,j,r}-\sum_{s=1}^{i-1}q_{s,j,r}y_{i-s,s+j,r}]$$
while $$q_{i,j,r}=\{p_{i,j,r}-\sum_{s=1}^{i-1}q_{s,j,r}y_{i-s,s+j,r}\}$$
are generalized polynomials with finite degree independent with $g$.

That is what we want:
 $$N(g\Gamma,U_{\ep})=\bigcup_{r=1}^b(\bigcap_{i,j}\{nb+r :|q_{i,j,r}(n)-q_{i,j,r}(0)|<\ep\})$$
Then the theorem is right for matrix nilmanifold.

\subsection{Proof of Theorem \ref{iso of torus and nilmanifold}}

The integer $b$ is given earlier. For each $1\leq r\leq b$, $(\overline{orb(\tau^rg,\tau^b)},\tau^b)$ is a dynamical system. Also we have $\overline{orb(g,\tau)}=\bigcup_{1\leq r\leq b}\overline{orb(\tau^rg,\tau^b)}$. The reason is, for each $x\in \overline{orb(g,\tau)}$, if $x_j\rightarrow_{j\in J} x$ with $x_j\in orb(g,\tau)$, then there is $r_x$ such that $\{x_j\}_{j\in J}\cap \overline{orb(\tau^{r_x}g,\tau^b)}$ are infinite set. Let $X_j=\overline{orb(\tau^rg,\tau^b)}$.

According to the proof of Theorem \ref{torus,Lie group,compact abelian group nilpotent} and Theorem \ref{F_d=F_GPd}, $(X_r,\tau^b)$ are systems which are AA system. According to Theorem \ref{nilfactor}, we know that the system is a almost one to one extension of its maximal nilfactor $(Y_r,S_r)$ with factor map $\pi_r$. Since $(X_r,\tau^b)$ is minimal and each point $x\in X_r$ is AA point, it is one to one extension. i.e. $(Y_r,S_r)=(X_r,\tau^b)$.

Now we construct a dynamical system $(Y,S)$. We consider each $Y_r $ as different dynamical systems with national isomorphic map $I_{r}:Y_r\rightarrow Y_{r+1}$. (The map is induced by $\tau:X_r\rightarrow X_{r+1}$). Also $I_d: Y_d\rightarrow Y_1$ is induced by $\tau:X_d\rightarrow X_1$. Let $Y=\bigsqcup_{r=1}^b Y_r$, $S|_{Y_r}=I_r$. Since each $(Y_r,S^b)$ is a system of order $d$, by the property of nilsystem, $(Y,S)$ is a system of order $d$.
Since $(\overline{orb(g,\tau)},\tau)$ is a factor of $(Y,S)$, it is a system of order $d$.

\subsection{Compact abelian groups}\label{section-group}

We also prove a similar but weaker conclusion for abelian compact group. That is:
\begin{thm}\label{compact abelian group main conclusion}
Let $X$ be a abelian compact group, and $\tau: X\rightarrow X, x\mapsto \a xx_0$ be an affine transformation with zero entropy. Where $\a$ is zero entropy automorphism. Then for each $x\in X$, $(\overline(orbit(x,\tau)),\tau)$ is isomorphic to an inverse limit of nilsystems.\end{thm}
Now let us consider a compact abelian group $X$ with zero entropy affine transformation.

The affine transformation of $X$ is $\tau: X\rightarrow X$, $\tau(g)=A(g)g_0$. Here $A$ is a continuous automorphism, $g_0\in X$. However, since we don't need the $A$-invariant metric and each system has an invertible fact, we can also use Property \ref{Lie chain} for $A$. And we need a property.
\begin{pro}\cite{Schmidt} \label{Lie chain}
Let $\a$ be an automorphism of a compact group $X$ with zero entropy. Then there exists a decreasing sequence of closed, normal, $\a$-invariant subgroups $$X=Y_0\supset Y_1\supset \ldots \supset Y_n\supset \ldots$$
such that $\bigcap_{i\geq 1}Y_i=\{1_X\}$, and $Y_{i-1}/Y_i$ is a compact Lie group with an $\a^{Y_{i-1}/Y_i}$-invariant metric for every $i\geq 1$.

In particular, there exists no infinite, closed, $\a$-invariant subgroup $Y\subset X$ such that $\a^Y$ is ergodic.
\end{pro}

Proof of {Theorem \ref{compact abelian group main conclusion}}

According to Property \ref{Lie chain}, for a compact abelian group $X$ with zero entropy automorphism $\a$, there exists a decreasing sequence of closed, normal, $\a$-invariant subgroups $$X=Y_0\supset Y_1\supset \ldots \supset Y_n\supset \ldots$$
such that $\bigcap_{i\geq 1}Y_i=\{1_X\}$, and $Y_{i-1}/Y_i$ is a compact Lie group with an $\a^{Y_{i-1}/Y_i}$-invariant metric for every $i\geq 1$.

For the zero entropy affine transformation $\tau$ of $G$, $\tau(x)=\a xx_0$ We notice that for the compact abelian Lie group $(X/Y_i,\a_i)$, where $\a_i: xY_i\rightarrow \a(x)x_0Y_i$ is a zero entropy affine transformation for $X/Y_i$, also the inverse limit of $(X/Y_i,\a_i)$ is $X$.

Also for each $x\in X$, $(\overline{orb(x,\tau)},\tau)$ is the inverse limit of $(\overline{orb(xY_i,\tau_i)},\tau_i)$. According to Theorem \ref{iso of torus and nilmanifold}, we know that each $(\overline{orb(xY_i,\tau_i)},\tau_i)$ is a system of order $d_i$. Then $(\overline{orb(x,\tau)},\tau)$ is a $\infty$-order system.

%\section{Sarnak Conjecture}\label{section-Sarnak}

%%%%%%%%%%%%%%%%%%%%%%%%%%%%%%%%%%%%%%%%%%%%%%%%%%%%%%%%%%%%%%%%%%%%%%%%%%%%%%%
%%%%%%%%%%%%%%%%%%%%%%%%%%%%%%%%%%%%%%%%%%%%%%%%%%%%%%%%%%%%%%%%%%%%%%%%%%%%%%%

\end{document}